\documentclass{mcom-l}

\usepackage{amssymb}

\newtheorem{theorem}{Theorem}
\newtheorem{lemma}{Lemma}

\theoremstyle{definition}

\theoremstyle{remark}

\numberwithin{equation}{section}

\begin{document}

\title[Tietze Extension in Constructive Mathematics]{Tietze Extension Does Not Always Work In Constructive Mathematics If Closed Sets Are Defined As Sequentially Closed Sets}


\author{Shun Ding}
\address{}
\curraddr{}
\email{}
\thanks{}

\author{Yang Wan}
\address{}
\curraddr{}
\email{}
\thanks{}

\author{Luofei Wang}

\author{Siqi Xiao}

\date{\today}

\dedicatory{}
\begin{abstract}
We prove that Tietze Extension does not always exist in constructive mathematics if closed sets on which the function we are extending are defined as sequentially closed sets. Firstly, we take a discrete metric space as our topological space. Now all sets are open and sequentially closed. Then, we form an unextendible algorithmic function transforming positive integers to 0 and 1, looking at the preimages of these values as our sequentially closed sets. Then we show that if the Tietze theorem conclusion holds for these closed sets then the unextendible function is extendible, thus giving us a contradiction.
\end{abstract}

\maketitle


\section*{Introduction}
Constructive mathematics differs from classical mathematics by interpreting "there exists" as "we can construct." All logical connectives and quantifiers are understood as constructive procedures. Notably, the Russian school allows the Markov Principle (if one can refute that a set is empty, one can find an element), while the American school does not \cite{bishop2012constructive,kushner1984lectures,markov1962constructive}.
Constructive real numbers (CRNs) were defined by Alan Turing (1936), who defined a real as constructive if there exists a \emph{computable} function
\[
f \colon \mathbb{N} \to \mathbb{Q}
\]
such that \cite{turing1936computable,turing1938computable}
\[
\forall k \in \mathbb{N}\; \bigl|x - f(k)\bigr| < 2^{-k}.
\]
This operational definition provided by him rejects the classical continuum and establishes \emph{computability} as the bedrock of existence \cite{turing1936computable,turing1938computable,bishop2012constructive}. Constructive functions mean that a function on constructive numbers
\[
f \colon \mathbb{R}_{\mathrm{c}} \to \mathbb{R}_{\mathrm{c}}
\]
can map CRNs to CRNs and it is algorithmic. Of course, equivalent CNRs should be mapped to equivalent CRNs.

Markov–Tseitin theorem says that all constructive functions are continuous, and Zaslavskii says that the closed bounded interval is not compact in the sense of open covers definition but is compact in terms of existence of finite $\varepsilon$‑net for each $\varepsilon$ \cite{kushner1984lectures}.

In classical mathematics, the \emph{Tietze Extension Theorem} is a fundamental result in topology, asserting that any real‑valued continuous function defined on a closed subset of a normal topological space can be extended to a continuous function on the whole space \cite{munkres2017topology,tietze1915funktionen}. Its proof relies non‑constructively on the Axiom of Choice and Law of Excluded Middle \cite{munkres2017topology}.

However, constructive mathematics adopts a stricter view on existence proofs: mathematical objects must be explicitly constructed by an algorithm, and proofs must avoid non‑constructive principles. This paradigm shift, led by Brouwer's intuitionism \cite{brouwer1907over}, Bishop's constructive analysis \cite{bishop2012constructive}, and also developed by Markov \cite{markov1962constructive} and Shanin \cite{shanin1962constructive}, reveals that many classical theorems fail under constructive scrutiny. In particular, the validity of extension theorems like Tietze's becomes highly sensitive to the precise definitions of topological concepts. A critical point of divergence arises in the definition of closed sets. While classical topology typically defines closed sets as complements of open sets, constructive approaches often employ alternative characterizations, such as \emph{sequentially closed sets}, to better align with computability and explicit definability. These two definitions of closed sets are equivalent in point‑set topology but not in constructive topology \cite{bishop2012constructive,kushner1984lectures}.

This paper investigates the status of the Tietze Extension Theorem in a constructive setting where closed sets are defined as sequentially closed. We demonstrate that, contrary to the classical case, the Tietze theorem does not always hold constructively under this definition. Our approach centers on a specific counterexample: we construct a metric space $X$ and a sequentially closed subset $A\subset X$ with a continuous function $f\colon A \to \mathbb{R}$ that admits no continuous extension to $X$.The counter example uses natural features of constructive logic, including the undecidability of disjunctions and the inability to uniformly decide convergence properties, which obstruct the extension process.

This result emphasizes a deeper tension between classical and constructive topology: definitions equivalent in classical topology may bifurcate in constructive settings, leading to divergent theorem validity. It also highlights the necessity of carefully reevaluating foundational tools when transitioning to constructive frameworks. Our work contributes to the broader program of constructive analysis by clarifying the limitations of extension theorems and emphasizing the role of definability in continuity principles.

\section*{Method}

\subsection*{Definition} A \emph{normal topological space} is a space in which any two disjoint closed sets have disjoint open neighbourhoods and, moreover, every singleton set is closed.

\bigskip
\noindent\textbf{A particular case of the Tietze Extension Theorem (two--closed--sets version).}  \\
Let $X$ be a normal topological space and let $A,B\subseteq X$ be disjoint closed subsets.  Any continuous function
\[
f\colon A\cup B \rightarrow \mathbb{R},
\qquad f|_{A}=0,\; f|_{B}=1,
\]
can be extended to a continuous function 
\[
F\colon X \rightarrow \mathbb{R}
\]
such that $F|_{A}=f|_{A}$ and $F|_{B}=f|_{B}$ \cite{munkres2017topology}.

\medskip
\noindent In particular, if $f$ maps $A$ to a constant $a\in\mathbb{R}$ and $B$ to a constant $b\in\mathbb{R}$, then there exists a continuous function $F\colon X\to\mathbb{R}$ with
\[
F(x)=a \quad\text{for all }x\in A, 
\qquad 
F(x)=b \quad\text{for all }x\in B .
\]

\bigskip
\noindent This last statement is the well‑known \emph{Urysohn Lemma}, which is needed to prove the general Tietze Extension Theorem in point‑set topology.

\bigskip
\subsection*{A discrete counterexample setup} 
Define a topological space $X=(\Omega,\tau)$ with $\Omega=\mathbb{N}$ and the discrete metric
\[
d(x,y)=
\begin{cases}
1,& x\neq y,\\
0,& x=y .
\end{cases}
\]
In this space $X$, every subset is both closed and open.
\begin{lemma}
Every subset of \(X\) is sequentially closed.
\end{lemma}

\begin{proof}
Take any set \(E \subseteq X\). If \(E = \varnothing\), then there is no sequence in the set, thus it is sequentially closed.
If \(E \neq \varnothing\), let \(\{x_n\}\) be a convergent sequence in \(E\) that converges to \(x\).
Since the metric is discrete, the sequence must stabilise after a certain moment with all its points identical.
Hence \(E\) contains the limit of the sequence. 
\end{proof}

\bigskip

\begin{theorem}
There exists a constructive function that cannot be extended to the whole space.
\end{theorem}

There exists a \emph{computable} function \(f\) on positive integers that takes values only \(0\) and \(1\) and that does not have an everywhere defined computable extension; this result, first proved by Turing, can be found in Vereshchagin--Shen (\cite{vereshchagin2003computable}, Theorem~8). Let \(f(x)\) denote this unextendible function.  Let
\[
A = \{x \mid f(x)=0\}, 
\qquad
B = \{x \mid f(x)=1\}.
\]
Both sets are sequentially closed and the function \(f\) is continuous when viewed as a function on \(A \cup B\). Consequently, \(f\) does \emph{not} admit a continuous computable extension to the whole space \(X\).

\textbf{Remark:}
Our topological space is normal. Indeed, take any closed \(C,D \subset X\) with \(C \cap D = \varnothing\).
Since every subset of \(X\) is open, the sets \(C' = C\) and \(D' = D\) are open and still disjoint.
Therefore, for any \(C\) and \(D\) we can find two disjoint neighbourhoods \(C'\) and \(D'\), so the space is normal.
It is, of course, also clear that every singleton set is closed.

\begin{proof}
Assume, for contradiction, that the conclusion of the Tietze theorem holds. Then the function \(f(x)\) can be extended to a continuous constructive function with values in constructive real numbers. Let
\[
F \colon X \rightarrow \mathrm{CRN}
\]
be such an extension. Let
\[
g \colon \text{All } \mathrm{CRN} \rightarrow [0,1]
\]
be a program that performs one step of the algorithm computing the rational approximation of a constructive real number and outputs the resulting rational \(a\).


If the conclusion of the Tietze theorem holds, then we can extend the function \(f(x)\) to a continuous constructive function whose values lie in the constructive real numbers.  
We argue by contradiction and let
\[
F \colon X \rightarrow \mathrm{CRN}
\]
be such an extension.

\medskip
Let
\[
g \colon \text{All }\mathrm{CRN} \rightarrow [0,1]
\]
be a program that performs one step of the algorithm computing a rational approximation of a constructive real number and outputs that rational as \(a\).
Define
\[
g(x)=
\begin{cases}
0, & a < 0.5,\\[4pt]
1, & a \ge 0.5.
\end{cases}
\]

\medskip
Set
\[
h(x)=g\bigl(F(x)\bigr),
\]
so \(h\) is the composition of two programs: it applies \(g\) to the result of \(F\).
The program \(h\) always terminates on all inputs. Moreover, \(h(x)\) is an extension of \(f(x)\) and takes only the values \(0\) and \(1\).

\medskip
Hence we obtain a continuous extension of an \emph{unextendible} function, which is a contradiction. 
\end{proof}

\textbf{Remark:}
The program \(g\) is \emph{not} well defined as a constructive function on the space of all \(\mathrm{CRN}\)s because different programs producing equivalent \(\mathrm{CRN}\)s may be mapped to different \(0,1\) values.
Nevertheless, the composition \(h\) is still well defined, so there is no gap in the proof.
Indeed, for two distinct programs computing the same constructive real number \(F(x)\), the corresponding inputs \(x\) must be different; hence for every \(x\) there is a unique value \(F(x)\) and a unique program that computes this \(\mathrm{CRN}\) (whether or not it has other equivalent representations).
Consequently, the first step of the program yields a unique output, and there is only one possible value $h(x)=g\bigl(F(x)\bigr)$ for each \(x\).

\bigskip
\textbf{Remark:}
Munkres presents two versions of the Tietze Extension Theorem: one where the range of the function is a bounded interval and another where the range is the entire real line.
Above we proved that the latter version fails in constructive mathematics when closed sets are interpreted as \emph{sequentially closed}.
By a similar argument, the bounded‑interval version of the Tietze theorem also fails in the constructive setting.
Urysohn's lemma states that if a topological space is \emph{normal}, then any two disjoint closed subsets can be separated by a continuous function. 
Indeed, Urysohn's lemma is a crucial step in proving the Tietze Extension Theorem.

\bigskip
Specifically, let \(X\) be a normal topological space and let \(A,B \subseteq X\) be disjoint closed subsets. 
Then there exists a continuous function
\[
f \colon X \rightarrow [0,1]
\]
such that
\[
f(a)=0 \quad\text{for all } a\in A,
\qquad
f(b)=1 \quad\text{for all } b\in B .
\]

\bigskip
Our result also shows that the conclusion of Urysohn's lemma does \emph{not} always hold in constructive mathematics when closed sets are defined as \emph{sequentially closed} sets.

\bigskip

\textbf{Open questions:}  
Does the Tietze theorem hold if we define closed sets as the complements of \emph{constructive open sets}?  
From our perspective, with this interpretation of closed sets the answer may differ from the case where closed sets are defined as \emph{sequentially closed} sets.

\bigskip
\textbf{Remark:}
A constructive open set is one for which, at every point, there exists a program that produces an open ball contained in the set and containing that point.  
\emph{Lacombe open sets} are a special class: they are unions of a computable sequence of rational open balls whose enumeration is effectively given and whose membership is semi‑decidable \cite{kushner1984lectures}.  
Thus every Lacombe open set is constructive, but not conversely.  
For example, in a discrete metric space every singleton is open, so one can take an \emph{unenumerable} union of points—yielding a constructive open set that is \emph{not} Lacombe.  
(An instance of an unenumerable set is the complement of an enumerable undecidable set; see Post's theorem of Vinogradov–Schen \cite{vereshchagin2003computable}.)

\bigskip
\textbf{Acknowledgement:}
We are grateful to the \emph{Neoscholar} company for organising the CIS programme during which this research was conducted.  
We thank \emph{Viktor Chernov} and \emph{Vladimir Chernov} for posing the question, and our TA \emph{Nan} for assistance throughout the programme.
\bibliographystyle{amsplain}
\bibliography{ref}

\providecommand{\bysame}{\leavevmode\hbox to3em{\hrulefill}\thinspace}
\providecommand{\MR}{\relax\ifhmode\unskip\space\fi MR }
\providecommand{\MRhref}[2]{%
  \href{http://www.ams.org/mathscinet-getitem?mr=#1}{#2}
}
\providecommand{\href}[2]{#2}
\begin{thebibliography}{10}

\bibitem{bishop2012constructive}
Errett Bishop and Douglas Bridges, \emph{Constructive analysis}, vol. 279, Springer Science \& Business Media, 2012.

\bibitem{brouwer1907over}
Luitzen Egbertus~Jan Brouwer, \emph{Over de grondslagen der wiskunde}, Maas \& van Suchtelen, 1907.

\bibitem{kushner1984lectures}
Boris~Abramovich Kushner and Lev~I\_Akovlevich Le\_fman, \emph{Lectures on constructive mathematical analysis}, vol.~60, American Mathematical Soc., 1984.

\bibitem{markov1962constructive}
Andrei~Andreevich Markov, \emph{On constructive mathematics}, Trudy Matematicheskogo Instituta imeni VA Steklova \textbf{67} (1962), 8--14.

\bibitem{munkres2017topology}
James Munkres, \emph{Topology (classic version)}, 2017.

\bibitem{shanin1962constructive}
Nikolai~Aleksandrovich Shanin, \emph{Constructive real numbers and constructive functional spaces}, Trudy Matematicheskogo Instituta imeni VA Steklova \textbf{67} (1962), 15--294.

\bibitem{tietze1915funktionen}
Heinrich Tietze, \emph{{\"U}ber funktionen, die auf einer abgeschlossenen menge stetig sind.},  (1915).

\bibitem{turing1938computable}
Alan~Mathison Turing, \emph{On computable numbers, with an application to the entscheidungsproblem. a correction}, Proceedings of the London Mathematical Society \textbf{2} (1938), no.~1, 544--546.

\bibitem{turing1936computable}
Alan~Mathison Turing et~al., \emph{On computable numbers, with an application to the entscheidungsproblem}, J. of Math \textbf{58} (1936), no.~345-363, 5.

\bibitem{vereshchagin2003computable}
Nikolai~Konstantinovich Vereshchagin and Alexander Shen, \emph{Computable functions}, vol.~19, American Mathematical Soc., 2003.

\end{thebibliography}

\end{document}